\documentclass[12pt,reqno]{amsart}
\usepackage{fullpage}
\usepackage{times}
\usepackage{amsmath,amssymb,amsthm,url}
\usepackage[utf8]{inputenc}
\usepackage[english]{babel}
\usepackage{bbm}
\usepackage{enumerate}
\usepackage{bm}
\usepackage{graphicx}
\usepackage[colorlinks=true, pdfstartview=FitH, linkcolor=blue, citecolor=blue, urlcolor=blue]{hyperref}

\newtheorem{thm}{Theorem}[section]
\newtheorem{lem}[thm]{Lemma}
\newtheorem{prop}[thm]{Proposition}
\newtheorem{defn}[thm]{Definition}

\newtheorem{cor}[thm]{Corollary}
\newtheorem{rem}[thm]{Remark}

\newtheorem{conj}[thm]{Conjecture}

\newcommand{\Tr}{\operatorname{Tr}}

\newcommand{\N}{\mathbb{N}}
\newcommand{\Z}{\mathbb{Z}}
\newcommand{\F}{\mathbb{F}}

\begin{document}

\title{Gauss sums and the maximum cliques in generalized Paley graphs of square order}
\author{Chi Hoi Yip}
\address{Department of Mathematics \\ University of British Columbia \\ 1984 Mathematics Road \\ Canada V6T 1Z2}
\email{kyleyip@math.ubc.ca}
\subjclass[2020]{Primary 11T24; Secondary 05C69, 11T30}
\keywords{Gauss sum, Paley graph, clique number, maximum clique.}
\date{\today}

\maketitle

\begin{abstract}
Let $GP(q,d)$ be the $d$-Paley graph defined on the finite field $\mathbb{F}_q$. It is notoriously difficult to improve the trivial upper bound $\sqrt{q}$ on the clique number of $GP(q,d)$. In this paper, we investigate the connection between Gauss sums over a finite field and the maximum cliques of their corresponding generalized Paley graphs. We show that the trivial upper bound on the clique number of $GP(q,d)$ is tight if and only if $d \mid (\sqrt{q}+1)$, which strengthens the previous related results by Broere-D\"oman-Ridley and Schneider-Silva. We also obtain a new simple proof of Stickelberger's theorem on evaluating semi-primitive Gauss sums. 
\end{abstract}

\section{Introduction}

Throughout the paper, we let $p$ be an odd prime, $q$ a power of $p$, $\F_q$ the finite field with $q$ elements, and $\F_q^*=\F_q \setminus \{0\}$. We always assume that $d$ is a divisor of $q-1$ such that $d>1$.

\subsection{Basic terminology}
Paley graphs and generalized Paley graphs form a nice and well-studied family of Cayley graphs.  The study of Paley graphs and generalized Paley graphs nicely connects many branches of mathematics, such as number theory, algebraic graph theory, discrete geometry, design theory, and coding theory \cite{GJ, thesis}.  In this paper, we will focus on improving the trivial upper bound on the clique number of a generalized Paley graph with square order using a number-theoretical approach. 

Suppose $q$ is a power of a prime $p$ such that $q \equiv 1 \pmod{4}$. The {\em Paley graph} on $\F_q$, denoted $P_q$, is the graph whose vertices are the elements of $\F_q$, such that two vertices are adjacent if and only if the difference of the two vertices is a square in $\F_q^*$. Similarly one can define generalized Paley graphs. They were first introduced by Cohen \cite{SC} in 1988, and later reintroduced by Lim and Praeger \cite{LP} in 2009. Let $d>1$ be a positive integer. The {\em $d$-Paley graph} on $\F_q$, denoted $GP(q,d)$, is the graph whose vertices are the elements of $\F_q$, where two vertices are adjacent if and only if the difference of the two vertices is a $d$-th power in $\F_q^*$. It is standard (see for example \cite[Section 4]{SC} and \cite[Section 5]{Yip2}) to further assume that $q \equiv 1 \pmod {2d}$. 

Recall that a {\em clique} in a graph $X$ is a subgraph of $X$ that is a complete graph. For a graph $X$, the {\em clique number} of $X$, denoted $\omega (X)$, is the size of a maximum clique of $X$. The structure of cliques in generalized Paley graphs is notoriously difficult to study \cite{CL}. In particular, finding reasonably good bounds for their clique numbers remains to be an open problem in additive combinatorics \cite{CL}, and there is a huge gap between the best-known upper bounds and lower bounds. The following upper bound on the clique number is well-known.

\begin{thm} [Trivial upper bound on the clique number] \label{trivial_ub}
Let $d$ be a positive integer greater than $1$. Let $q \equiv 1 \pmod {2d}$ be a prime power, then  $\omega\big(GP(q,d)\big) \leq \sqrt{q}$. 
\end{thm}

This square root upper bound on the clique number is referred to as the {\em trivial upper bound} in the literature \cite{BMR, HP, MP, Yip2}; see \cite[Chapters 2]{thesis} for a survey of different proofs of Theorem \ref{trivial_ub} using tools from combinatorics and algebraic graph theory. For the sake of completeness, a simple proof of Theorem \ref{trivial_ub} will be given in Section \ref{PGC}. Throughout the paper, whenever we refer to the trivial upper bound, we are always referring to this square root upper bound on the clique number of a given generalized Paley graph. 

The current best-known upper bound on the clique number of a (generalized) Paley graph of order $q$ is $O(\sqrt{q})$, which is the same as the order of magnitude of the trivial upper bound. There were some very recent improvements on $\omega\big(GP(q,d)\big)$ for non-squares $q$ using polynomial methods. However, essentially there were no known improvements on $\omega\big(GP(q,d)\big)$ for squares $q$. See Section \ref{prev} for a survey of the previous works.

\subsection{Main results}
To improve the trivial upper bound, it is crucial to distinguish the square cases from the non-square case. In the case that $q$ is a non-square, $\sqrt{q}$ is not an integer and thus the trivial upper bound is never tight. While in the case that $q$ is a square, the subfield $\F_{\sqrt{q}}$ may form a maximum clique and consequently $\omega\big(GP(q,d)\big)$ attains the trivial upper bound; see also the discussion in Section \ref{algcon}. This phenomenon is known as an example of the so-called \emph{subfield obstruction}; see also Section \ref{PGC}. Thus, in the case that $q$ is a square, any attempt to improve the trivial upper bound must be able to overcome the subfield obstruction.

The following is our main result, which gives a necessary and sufficient condition for which the trivial upper bound on the clique being tight. 

\begin{thm}\label{main}
Let $d$ be a positive integer greater than $1$. Let $q \equiv 1 \pmod {2d}$ be an even power of a prime $p$. Then $\omega\big(GP(q,d)\big)=\sqrt{q}$ if and only if $d \mid (\sqrt{q}+1)$. In particular, if $d \nmid (\sqrt{q}+1)$, then $\omega\big(GP(q,d)\big) \leq \sqrt{q}-1$.
\end{thm}

Theorem \ref{main}, which will be proved in Section \ref{proof}, strengthens the previous related results by Broere-D\"oman-Ridley and Schneider-Silva; see Theorem \ref{t4} and Theorem \ref{tss}. Note that if $d \nmid (\sqrt{q}+1)$, we managed to get a non-trivial upper bound on the clique number. We will see in Section \ref{prev} that even improving the trivial upper bound by 1 is very difficult, especially if $q$ is a square. 

One key ingredient of the proof of Theorem \ref{main} is Theorem \ref{main2}, which gives a Fourier analytic characterization of maximum cliques in a generalized Paley graph. Before stating Theorem \ref{main2}, we recall the definition of Fourier transforms over the finite field $\F_q$. For any prime $p$, and any $x \in \F_p$, we follow the standard notation that $e_p(x)=e(2 \pi i x/p)$, where we embed $\F_p$ into $\Z$. 

\begin{defn}
If $f:\F_q \to \mathbb{C}$, for each $c \in \F_q$, we define
$$
\hat{f}(c):=\frac{1}{q} \sum_{x \in \F_q} e_p\big(\Tr_{\F_q/\F_p}(cx)\big) f(x).
$$
\end{defn}

We will focus on the case that $f=\mathbbm{1}_A$, the indicator function of a subset $A$ of $\F_q$. In this case, the Fourier coefficient $\widehat{\mathbbm{1}_A}(c)$ is simply an exponential sum over $A$.

\begin{thm} \label{main2}
Let $d$ be a positive integer greater than $1$. Let $q \equiv 1 \pmod {2d}$ be an even power of a prime $p$ and let $\chi$ be a multiplicative character of $\F_q$ with order $d$. Let $A$ be a subset of $\F_q$ with $|A|=\sqrt{q}$. Then $A$ is a maximum clique in the generalized Paley graph $GP(q,d)$ if and only if for any $c \in \F_q^*$, the Gauss sum $G(\chi)$ and $\chi(c) |\widehat{\mathbbm{1}_A}(c)|^2$ share the same argument (as complex numbers).
\end{thm}

We will introduce $G(\chi)$ in Definition \ref{defn:Gauss sum}. In Section \ref{main2proof}, we will see that the proof of Theorem \ref{main2} only requires basic properties of trace functions and Gauss sums.

Theorem \ref{main2} provides a nice connection between Gauss sums and generalized Paley graph.  As a quick application of Theorem \ref{main2}, we give a new simple proof of the classical Stickelberger's theorem on explicit formulas of semi-primitive Gauss sums (see Theorem \ref{stick1}), which may be of independent interest to some readers. Theorem \ref{main2} also suggests a new approach to prove the known classification of maximum cliques in $GP(q,d)$, where $d \mid (\sqrt{q}+1)$; see Section \ref{embed} and Section \ref{maxclique}.

\subsection{Previous works on improving the trivial upper bound} \label{prev}
There was practically no improvement on the trivial upper bound until 2006, where Maistrelli and Penman \cite{MP} showed that if $q$ is a non-square such that $q \equiv 1 \pmod 4$, and $q>5$, then $\omega(P_q) \leq \sqrt{q-4}$. In 2013, Bachoc, Matolcsi, and Ruzsa \cite{BMR} showed that $\omega(P_q) \leq \sqrt{q}-1$ for half of the non-squares $q$. In 2019, Hanson and Petridis \cite{HP} used Stepanov's method to show that $\omega\big(GP(p,d)\big)\leq \sqrt{p/d}+O(1)$ for any $d \mid (p-1)$. Later, the author \cite{Yip2} were able to generalize their method to finite fields by analyzing binomial coefficients. In \cite[Section 5.3]{Yip2}, we proved the following theorem, which potentially leads to an improved upper bound on the clique number of any generalized Paley graph. 

\begin{thm}[\cite{Yip2}] \label{t5}
If $q\equiv 1 \pmod{2d}$, and $2 \leq n\leq N=\omega\big(GP(q,d)\big)$ satisfies
$
\binom{n-1+\frac{q-1}{d}}{\frac{q-1}{d}}\not \equiv 0 \pmod p,
$
then $(N-1)n \leq \frac{q-1}{d}$.
\end{thm}

To apply Theorem \ref{t5}, we need to understand the base-$p$ representation of $\frac{q-1}{d}$, or equivalently, the the order of $p$ modulo $d$. Once the order is determined, we can perfrom an ad-hoc analysis. For example, we can use Theorem \ref{t5} to show that $\omega\big(GP(q,d)\big)\leq \sqrt{q/d}(1+o(1))$ for any $d \mid (q-1)$ and non-square $q$, which is almost as good as Hanson-Petridis bound; see \cite[Theorem 5.10]{Yip2}. However, when the number of divisors of $\phi(d)$ is large, the analysis will become very complicated.

In the case that $q$ is a square, we have seen that the $\sqrt{q}$ bound might be tight, and thus the $\sqrt{q}$ bound is unlikely to be improved using this method in general. Nevertheless, in the special case that $d\geq 3$ and $d \mid (p-1)$, we used Theorem \ref{t5} to prove that $\omega\big(GP(q,d)\big) < \sqrt{\frac{q}{d}} \big(1+O(d^{-1/2})\big)$. Note that the conditions $d \geq 3$ and $d \mid (p-1)$ imply that $d \mid (\sqrt{q}-1)$, and thus $d \nmid (\sqrt{q}+1)$. Therefore, this is consistent with Theorem \ref{main} and shows that there is a space to improve the trivial upper bound when $d \nmid (\sqrt{q}+1)$.

Another approach to improve the trivial upper bound is to consider the number of directions determined by $C \times C$ in the affine Galois plane $AG(2,q)$, where $C$ is a clique in $GP(q,d)$. Recently,  Di Benedetto,  Solymosi, and White \cite{DSW} showed that $A\times B\subset AG(2,p)$ determines at least $|A||B| - \min\{|A|,|B|\} + 2$ directions, provided $|A||B|<p$. Using this result, they were able to recover the Hanson-Petridis bound on $GP(p,d)$.  In \cite{Yip2}, the author gave some generalization of this direction result to $AG(2,q)$ and showed that if $d \geq 3$, then for any function $h(x)=o(x)$, the bound $\omega\big(GP(p^{2r+1},d)\big) \leq p^{r+1/2}-h(p)$ holds for almost all primes $p$ such that $p^{2r+1} \equiv 1 \pmod {2d}$. However, this improved upper bound depends on the uniform distribution of the fractional part of $p^{r+1/2}$. If $q$ is a square, then $\sqrt{q}$ is always an integer and this approach fails to provide any improvement on the trivial upper bound.

\subsection{Structure of the paper}
The paper is organized as follows. In Section \ref{background}, we provide additional background and introduce several independent techniques to study maximum cliques in generalized Paley graphs. In Section \ref{nota}, we will introduce trace functions, Gauss sums, and their basic properties. In Section \ref{prep}, we will prove Theorem \ref{main} and Theorem \ref{main2}. We will also give a simple proof of Stickelberger's theorem and provide a characterization of maximum cliques involving only trace functions.

\section{Background and overview of the paper} \label{background}
In this section, we put our main results into context. We introduce several important yet independent tools to study generalized Paley graphs. A special case of Theorem \ref{main} will be proved in Corollary \ref{even}. In Section \ref{prep}, we will combine these tools (together with classical results on Gauss sums in Section \ref{nota}) to prove Theorem \ref{main}. 

\subsection{Algebraic constructions for the cliques} \label{algcon}
Numerical data for primes $p < 10000$ by Exoo~\cite{EX} suggests that $\omega(P_p)$ behaves like a polylogarithmic function, and it is widely believed $\omega(P_p) \ll_{\epsilon} p^{\epsilon}$ for any $\epsilon>0$. Some people even conjectured $\omega(P_p)=O(\log p \log \log p)$; see the discussion on these bounds as well as the lower bounds on the clique number in \cite{SC} and \cite[Section 3.5]{thesis}. 

More generally, for a given generalized Paley graph $GP(q,d)$, it is believed that $\omega\big(GP(q,d)\big)$ is bounded by a polylogarithmic function in $q$ unless we have some special algebraic constructions for the cliques in $GP(q,d)$; see \cite{Green} for a related discussion on the clique number of Cayley graphs defined on a cyclic group.

If $q$ is a proper prime power, then the finite field $\F_q$ contains some proper subfields. For the lower bound of the clique number, it is often helpful to consider whether a particular subfield of $\F_q$ forms a clique in $GP(q,d)$; see also \cite{AY, Yip3} for explicit subfield constructions of maximal cliques in certain generalized Paley graphs.  The following theorem is due to Broere, D\"oman, and Ridley.

\begin{thm}[\cite{BDR}]\label{t4}
Let $d$ be a positive integer greater than $1$. Let $q \equiv 1 \pmod {2d}$ be a power of an odd prime $p$. If $k$ is an integer such that $d \mid \frac{q-1}{p^k-1}$, then the subfield $\F_{p^k}$ forms a clique in $GP(q,d)$. In particular, if $q$ is a square and $d \mid (\sqrt{q}+1)$, then $\omega\big(GP(q,d)\big)=\sqrt{q}$.
\end{thm}

They proved Theorem \ref{t4} by showing that if $d \mid (\sqrt{q}+1)$, then the subfield $\F_{\sqrt{q}}$ forms a clique in $GP(q,d)$ and thus the clique number of $GP(q,d)$ is $\sqrt{q}$. The following lemma gives a necessary and sufficient condition for $\F_{\sqrt{q}}$ to form a clique in $GP(q,d)$. 

\begin{lem}\label{subfield}
Let $d$ be a positive integer greater than $1$. Let $q \equiv 1 \pmod {2d}$ be an even power of a prime $p$. The subfield $\F_{\sqrt{q}}$ forms a clique in $GP(q,d)$ if and only if $d \mid (\sqrt{q}+1)$.
\end{lem}

\begin{proof}
By definition, the subfield $\F_{\sqrt{q}}$ forms a clique in $GP(q,d)$ if and only if $a-b$ is a $d$-th power in $\F_q^*$, whenever $a,b \in \F_{\sqrt{q}}$ and $a \neq b$, which is equivalent to any nonzero element in $\F_{\sqrt{q}}$ is a $d$-th power in $\F_q^*$ since the subfield $\F_{\sqrt{q}}$ is closed under subtraction, and $0 \in \F_{\sqrt{q}}$. Let $g$ be the primitive root in $\F_{\sqrt{q}}$, then the order of $g$ is $\sqrt{q}-1$. Since $\F_{\sqrt{q}}^*$ is generated by $g$, the subfield $\F_{\sqrt{q}}$ forms a clique in $GP(q,d)$ if and only if $g$ is a $d$-th power in $\F_q^*$, which is equivalent to $g^{(q-1)/d}=1$, and also equivalent to $(\sqrt{q}-1) \mid \frac{q-1}{d}$, i.e., $d \mid (\sqrt{q}+1)$. 
\end{proof}

Recall that the {\em chromatic number} of a graph $X$, denoted $\chi(X)$, is the smallest number of colors needed to color the vertices of $X$ so that no two adjacent vertices share the same color. In fact, Broere, D\"oman, and Ridley \cite{BDR} showed that if $d \mid (\sqrt{q}+1)$, then both the chromatic number and the clique number of $GP(q,d)$ is $\sqrt{q}$. In \cite{SS}, Schneider and Silva studied generalized Paley graphs whose clique and chromatic numbers coincide using the synchronization property for permutation groups, and one of their main results is the following.

\begin{thm}[\cite{SS}] \label{tss}
Let $d$ be a positive integer greater than $1$. Let $q \equiv 1 \pmod {2d}$ be an even power of a prime $p$. Then $\chi\big(GP(q,d) \big)=\omega\big(GP(q,d)\big)=\sqrt{q}$ if and only if $d \mid (\sqrt{q}+1)$.
\end{thm}

Intuitively, the trivial upper bound is tight only if the maximum cliques of $GP(q,d)$ have some special algebraic structure. Given Theorem \ref{t4}, Lemma \ref{subfield}, and Theorem \ref{tss}, it is tempting to conjecture that this special algebraic structure is the subfield structure. In other words, the heuristic is that the trivial upper bound on $\omega\big(GP(q,d)\big)$ is tight if and only if the subfield $\F_{\sqrt{q}}$ forms a clique in $GP(q,d)$. Our main result, Theorem \ref{main}, confirms this heuristic.

Note that Theorem \ref{main} strengthens Theorem \ref{t4}, as we give a necessary and sufficient condition for which the trivial upper bound is tight. Theorem \ref{main} also refines the result Theorem \ref{tss}, since it removes the assumption on the chromatic number.

\subsection{Character sums and the Paley Graph Conjecture}\label{PGC}
The starting point to prove Theorem \ref{main} is to understand a finite field analog of the Paley Graph Conjecture, as well as the character sum approach to recover the trivial upper bound on the clique number of a given generalized Paley graph.

Let $A,B$ be arbitrary subsets of the field  $\F_p$, and $\chi$ a non-principal Dirichlet character modulo $p$. Many authors have studied the following double character sum (see for example \cite{Chung, ES1, Kar})
\begin{equation}\label{f:def_sum}
    \sum_{a\in A,\, b\in B}\chi(a+b) \,,
\end{equation}
and their goal is to obtain good upper bounds on its absolute value. The following classical estimate is  due to Erd\H{o}s and Shapiro \cite{ES1}.

\begin{thm}[\cite{ES1}] \label{ES1}
If $\chi$ is a non-principal Dirichlet character modulo $p$, then for any $A,B \subset \F_p$, 
 $$
 \bigg|\sum_{a\in A,\, b\in B}\chi(a+b)\bigg|  \leq \sqrt{p|A||B|}.
 $$
\end{thm}

In \cite{Chung}, Chung improved the upper bound in Theorem \ref{ES1} by a multiplicative factor.

\begin{thm}[{\cite[Theorem 2.1]{Chung}}] \label{Chung1}
If $\chi$ is a non-principal Dirichlet character modulo $p$, then for any $A,B \subset \F_p$, 
 $$
 \bigg|\sum_{a\in A,\, b\in B}\chi(a+b)\bigg|  \leq \sqrt{p|A||B|}\bigg(1-\frac{|A|}{p}\bigg)^{1/2}\bigg(1-\frac{|B|}{p}\bigg)^{1/2}.
 $$
\end{thm}

For our purpose, we establish the following generalization of Theorem \ref{Chung1} to any finite field. Note that we need to replace the notion of Dirichlet characters modulo a prime by multiplicative characters in a finite field. Recall a character $\chi$ of the multiplicative group $\F_q^*$ of $\F_q$ is called a {\em multiplicative character} of $\F_q$. It is a custom to define $\chi(0)=0$. For a multiplicative character $\chi$, its order $d$ is the smallest positive integer such that $\chi^d=\chi_0$, where $\chi_0$ is the trivial multiplicative character of $\F_q$. Note that the order $d \mid (q-1)$, and we assume $d>1$ to avoid the trivial multiplicative character. Similar to the statement of Theorem \ref{Chung1}, it is necessary to assume that $\chi$ is non-trivial.

\begin{thm} \label{t1.3}
If $\chi$ is a non-trivial multiplicative character of $\F_q$, then for any $A,B \subset \F_q$, 
 $$
 \bigg|\sum_{a\in A,\, b\in B}\chi(a+b)\bigg|  \leq \sqrt{q|A||B|}\bigg(1-\frac{|A|}{q}\bigg)^{1/2}\bigg(1-\frac{|B|}{q}\bigg)^{1/2}.
 $$
\end{thm}

The Paley Graph Conjecture is a well-known hypothesis on sums of the form (\ref{f:def_sum}); see for example \cite[Conjecture 2.2]{Chung}. 

\begin{conj}[Paley Graph Conjecture] \label{PC}
  Let $\delta>0$, and let $A,B\subset\mathbb{F}_p$ be arbitrary sets with  $|A|>p^{\delta}$ and $|B|>p^\delta$. Then there exist $c(\delta), \tau(\delta)>0$ such that for any sufficiently
 large prime number $p$ and any non-principal Dirichlet character $\chi$ modulo $p$,
\begin{equation} \label{Paleya}
    \bigg|\sum_{a\in A,\, b\in B}\chi(a+b)\bigg| \leq c(\delta) p^{-\tau}|A||B| \,.
\end{equation}
\end{conj}

For our purpose, we should instead consider the finite field analog of Conjecture \ref{PC}. The case $\delta>1/2$ has been confirmed by Theorem \ref{t1.3}. However, as Chung \cite{Chung} remarked, the case $\delta< 1/2$ will no longer hold due to subfield obstructions. For example, if $A=B$ are proper subfields of $\F_q$, and $\chi$ is a non-trivial multiplicative character such that $\chi|_A$ is trivial, then inequality \eqref{Paleya} fails to hold.

To see the connection between the Paley Graph Conjecture and the clique number of generalized Paley graphs, we associated a $d$-Paley graph $GP(q,d)$ with a multiplicative character in $\F_q$ with order $d$. In particular, we can apply Theorem \ref{t1.3} to recover Theorem \ref{trivial_ub}.

\begin{proof}[Proof of Theorem \ref{trivial_ub}]
Let $\chi$ be a multiplicative character of $\F_q$, with order $d$. Let $A$ be a maximum clique in $GP(q,d)$, and let $B=-A$. Then for any $a \in A, b \in B$, either $a+b=0$ or $a+b$ is a $d$-th power in $\F_q^*$; in the latter case, we have $\chi(a+b)=1$. Therefore, 
$$\sum_{a\in A,\, b\in B}\chi(a+b)=|A|^2-|A|.$$ 
Combining Theorem \ref{t1.3},  we get $|A|^2-|A| \leq \sqrt{q} |A|(1-\frac{|A|}{q})$, which implies $\omega\big(GP(q,d)\big)=|A|\leq \sqrt{q}$.
\end{proof}

We remarked that several different authors have succeeded in using character sums to study generalized Paley graphs; see for example \cite{AY, SC}. The character sum approach to tackle the estimates of clique number looks promising since Conjecture \ref{PC} would immediately imply a much better upper bound. Unfortunately, at the moment we do not understand Conjecture \ref{PC} very well; even the case $|A|\sim |B|\sim p^{\frac12}$ for inequality~(\ref{Paleya}) is unknown (see \cite{Kar} for a related discussion). 

\subsection{Embeddings among generalized Paley graphs}\label{embed}
The following simple observation is crucial for the proof of our main results.

\begin{lem}\label{embedding}
Let $GP(q,d)$ and $GP(q,d')$ be generalized Paley graphs. If $d' \mid d$, then $GP(q,d)$ is a subgraph of $GP(q,d')$; in particular, if $\omega\big(GP(q,d)\big)=\sqrt{q}$, then $\omega\big(GP(q,d')\big)=\sqrt{q}$, and every maximum clique in $GP(q,d)$ is also a maximum clique in $GP(q,d')$.
\end{lem}
\begin{proof}
Suppose $d' \mid d$. Then every $d$-th power in $\F_q^*$ is also a $d'$-th power, and thus $GP(q,d)$ is a subgraph of $GP(q,d')$ follows from the definition of generalized Paley graphs. Therefore, $\omega\big(GP(q,d)\big)\leq \omega\big(GP(q,d')\big) \leq \sqrt{q}$, and any clique in $GP(q,d)$ is a clique in $GP(q,d')$. In particular, if $\omega\big(GP(q,d)\big)=\sqrt{q}$, then $\omega\big(GP(q,d)\big)=\omega\big(GP(q,d')\big)=\sqrt{q}$, so every maximum clique in $GP(q,d)$ is also a maximum clique in $GP(q,d')$. 
\end{proof}

If $q$ is a square and $d \mid (\sqrt{q}+1)$, the following theorem characterizes all maximum cliques in $GP(q,d)$.

\begin{thm}[{\cite[Theorem 1.2]{PS}}] \label{PS}
If $q$ is a square and $d$ is a divisor of $(\sqrt{q}+1)$ such that $d>1$, then in the generalized Paley graph $GP(q,d)$, the only maximum clique containing $0,1$ is the subfield $\F_{\sqrt{q}}$; in other words, any maximum clique in $GP(q,d)$ is an affine transformation of $\F_{\sqrt{q}}$.
\end{thm}

The proof of Theorem \ref{PS} used discrete geometry and polynomial method, in particular a careful study of symmetric polynomials. Theorem \ref{PS} was first proved by Blokhuis \cite{AB} for Paley graphs. Later, Sziklai \cite{PS} generalized Blokhuis' proof. 

It is interesting to produce such a proof using tools from algebraic graph theory. Paley graphs are instances of strongly regular graphs and there are many algebraic graph theory approaches to recover the trivial upper bound on the clique number. Thus, it is natural to ask when is the trivial upper bound tight, or equivalently, how to classify the maximum cliques. Godsil and Meagher \cite{GM} suggest an algebraic approach to prove Theorem \ref{PS} using eigenvalues and eigenspaces in the setting of applying ratio bounds to a strongly regular graph; see the discussion in \cite[Section 5.9]{GM}. 

We are most interested in finding a proof of Theorem \ref{PS} using character sums. Recently, Asgarli and the author \cite{AY} were able to prove a slightly weaker version of Theorem \ref{PS} using classical tools on the direction set determined by a point set and character sum estimates over subspaces. Theorem \ref{main2} also suggests a Fourier analytic approach to prove Theorem \ref{PS}, however, at this moment we are not aware of how to make use of this new approach; see Section \ref{maxclique} for an equivalent characterization involving only the trace functions.

We finish the section by showing that if $d$ is even, then Theorem \ref{main} holds. The proof is a nice combination of Lemma \ref{embedding} and Theorem \ref{PS}.

\begin{cor}\label{even}
Let $q$ be a square such that $q \equiv 1 \pmod {2d}$ and $d \nmid (\sqrt{q}+1)$. If $\gcd(d, \sqrt{q}+1)>1$, then $\omega\big(GP(q,d)\big)\leq \sqrt{q}-1$. In particular, if $d$ is even, then $\omega\big(GP(q,d)\big)\leq \sqrt{q}-1$.
\end{cor}
\begin{proof}
Let $d'=\gcd(d,\sqrt{q}+1)$. Assume $d'>1$; then $d' \mid d$, $d' \mid (\sqrt{q}+1)$, and the generalized Paley graph $GP(q,d')$ is well-defined. Suppose $\omega\big(GP(q,d)\big)=\sqrt{q}$. Let $C$ be a maximum clique in $GP(q,d)$. Without loss of generality, we may assume that $0 \in C$, and thus each nonzero element in $C$ is a $d$-th power in $\F_q^*$. Let $x \in C \setminus \{0\}$, then $C'=x^{-1}C$ is still a (maximum) clique in $GP(q,d)$; moreover, $0,1 \in C'$. By Lemma \ref{embedding}, $C'$ is also a maximum clique in $GP(q,d')$. It follows from Theorem \ref{PS} that $C'$ is the subfield $\F_{\sqrt{q}}$, i.e., the subfield $\F_{\sqrt{q}}$ forms a clique in $GP(q,d)$, which contradicts the assumption $d \nmid (\sqrt{q}+1)$ in view of Lemma \ref{subfield}.
\end{proof}

We remark that Corollary \ref{even} cannot be used to give a complete proof of Theorem \ref{main} in the case that $d$ is odd. Let $d$ be an odd divisor of $(\sqrt{q}-1)$, then $q \equiv 1 \pmod {2d}$, and thus the generalized Paley graph $GP(q,d)$ is well-defined. However, in this case, $\gcd(d, \sqrt{q}+1)=1$. To complete the proof of Theorem \ref{main}, in the next two sections we explore the connection between Gauss sums and generalized Paley graphs.

\section{Trace functions and Gauss sums}\label{nota}

\subsection{Trace functions and Fourier transforms}\label{trace}
For $\alpha \in F=\F_{q^m}$ and $K=\F_q$, the {\em trace} $\Tr_{F/K}(\alpha)$ of $\alpha$ over $K$ is defined by
$$
\Tr_{F/K}(\alpha)=\alpha+\alpha^q+\cdots+\alpha^{q^{m-1}}.
$$
If $K$ is the prime subfield of $F$, then $\Tr_{F/K}(\alpha)$ is called the {\em absolute trace} of $\alpha$, and simply denoted by $\Tr_{F}(\alpha)$. The following lemma lists some basic properties of the trace function.
\begin{lem} [{\cite[Theorem 2.23]{LN}}] \label{basic}
Let $K=\F_{q}$ and $F=\F_{q^{m}}$. Then the trace function $\Tr_{F / K}$ satisfies the following properties:
\begin{enumerate}[(i)]
    \item $\Tr_{F / K}(\alpha+\beta)=\Tr_{F / K}(\alpha)+\Tr_{F / K}(\beta)$ for all $\alpha, \beta \in F$;
    \item$\Tr_{F / K}(c \alpha)=c \Tr_{F / K}(\alpha)$ for all $c \in K, \alpha \in F$;
    \item  $\Tr_{F / K}$ is a linear transformation from $F$ onto $K,$ where both $F$ and $K$ are viewed as vector spaces over $K$.
\end{enumerate}
\end{lem}

The following lemma is about the transitivity of the trace function.

\begin{lem}[{\cite[Theorem 2.26]{LN}}] \label{trans}
Let $K$ be a finite field, let $F$ be a finite extension of $K$ and $E$ a finite extension of $F$. Then for any $a \in E$, 
$$
\Tr_{E/K}(a)=\Tr_{F/K}\big(\Tr_{E/F}(a)\big).
$$
\end{lem}

We introduce the following weighted version of the Fourier coefficients, which is helpful for our discussion.

\begin{defn}
Let $A$ be a subset of $\F_q$. For each $c \in \F_q^*$, we define 
\[
S(q,A;c)=\sum_{a\in A} e_p\big(\Tr_{\F_q}(ac)\big)=q\widehat{\mathbbm{1}_A}(c).
\]
\end{defn}

To compute the above exponential sums over a subset, it is helpful to establish the following orthogonality relations in finite fields. Recall that recall the orthogonality relation in the prime field $\F_p$ is the following: for $c \in \F_p$, we have
$$
\sum_{a \in \F_p} e_p(ca)=\begin{cases}
p, \quad \text{if }c=0, \\
0, \quad \text{if } c \in \F_p^*.
\end{cases}
$$
Note that this is just the formula for summing a geometric sequence. The following lemma extends the orthogonality relations to an arbitrary finite field. 

\begin{lem} \label{sumq}
Let $K$ be a subfield of $\F_q$. Then for $c \in \F_q$, $S(q,K;c)\neq 0$ if and only if  $\Tr_{\F_q/K}(c)=0$. In particular, if $c \in \F_q^*$, then $S(q,\F_q;c)=0.$
\end{lem}

\begin{proof}
 Let $c \in \F_q$. Note that $K$ is a finite extension of the prime field $\F_p$, and $\F_q$ is a finite extension of the subfield $K$. By Lemma \ref{trans},
$$
S(q,K;c)=\sum_{a\in K} e_p\big(\Tr_{\F_{q}/\F_p}(ac)\big)=\sum_{a\in K} e_p\bigg(\Tr_{K/\F_p}\big(\Tr_{\F_{q}/K}(ac)\big)\bigg).
$$
Since $a \in K$, by Lemma \ref{basic} (ii), $\Tr_{\F_{q}/K}(ac)=a\Tr_{\F_{q}/K}(c)$. If $\Tr_{\F_q/K}(c)=0$, then it is clear that $S(q,K;c)=|K|>0$.

Next we assume $\Tr_{\F_q/K}(c)\neq 0$. By Lemma \ref{basic} (iii), $\Tr_{\F_q/K}(c) \in K$. Note that as $a$ runs over $K$, $\Tr_{\F_{q}/K}(ac)=a\Tr_{\F_{q}/K}(c)$ also runs over $K$. Therefore, 
\begin{align*}
S(q,K;c) =\sum_{a\in K} e_p\bigg(\Tr_{K/\F_p}\big(\Tr_{\F_{q}/K}(ac)\big)\bigg)=\sum_{a\in K} e_p\big(\Tr_{K/\F_p}(a)\big).
\end{align*}
Note that $\Tr_{K/\F_p}(x)=0$ has at most $|K|/p$ solutions in $K$, so there exists $y \in K$ such that $\Tr_{K/\F_p}(y) \neq 0$,  i.e. $e_p\big(\Tr_{K/\F_p}(y)\big) \neq 1$. Observe that as $a$ runs over $K$, $a+y$ also runs over $K$, so by Lemma \ref{basic} (i),  we have
$$
S(q,K;c)=\sum_{a\in K} e_p\big(\Tr_{K/\F_p}(a+y)\big)
=e_p\big(\Tr_{K/\F_p}(y)\big) \sum_{a\in K} e_p\big(\Tr_{K/\F_p}(a)\big)=e_p\big(\Tr_{K/\F_p}(y)\big)S(q,K;c).
$$
It follows that $S(q,K;c)=0$.  In particular, if $K=\F_q$ and $c \in \F_q^*$, then $\Tr_{\F_q/K}(c)=c \neq 0$, and thus $S(q,\F_q;c)=0.$
\end{proof}

\subsection{Gauss sums over finite fields}\label{Gauss}
In this section, we recall some basic facts for Gauss sums over finite fields. We refer the reader to \cite{BEW} and \cite[Chapter 5]{LN} for more details. 

\begin{defn} \label{defn:Gauss sum}
Let $\chi$ be a multiplicative character of $\F_q$; then the Gauss sum associated to $\chi$ is defined to be $$G(\chi)=\sum_{ c \in \F_q} \chi(c) e_p\big(\Tr_{\F_q}(c)\big).$$ 
\end{defn}

Gauss sums are important and fundamental tools in number theory, and they also have rich applications in combinatorics, for example, they could be used to construct new partial difference sets \cite{BMW, KM} and strongly regular Cayley graphs \cite{FX}. The following are some basic properties of Gauss sums.

\begin{lem}[Theorem 5.11, \cite{LN}] \label{gq}
Let $\chi$ be a non-trivial multiplicative character of $\F_q$. Then $|G(\chi)|=\sqrt{q}$.
\end{lem}

\begin{defn}
Given a nontrivial multiplicative character $\chi$ of $\F_q$, define the normalized Gauss sum associated to $\chi$ to be $\epsilon(\chi)=q^{-1/2}G(\chi)$.
\end{defn} 
Given Lemma \ref{gq}, to determine $G(\chi)$, it suffices to determine $\epsilon(\chi)$, which lies on the unit circle. The following lemma is the key in proving Theorem \ref{t1.3}. It enables one to express $\overline{\chi(a)}$ as the linear combination of $\chi(c)$, where $c$ runs over $\F_q$.

\begin{lem} [Theorem 5.12, \cite{LN}] \label{qf}
Let $\chi$ be a multiplicative character of  $\F_q$. Then for any $a \in \F_q$, 
$$
\overline{\chi(a)}=\frac{1}{G(\chi)} \sum_{ c \in \F_q} \chi(c) e_p\big(\Tr_{\F_q}(ac)\big). 
$$
\end{lem}

The following is the definition of a pure Gauss sum.
\begin{defn} \label{pure1}
Let $\chi$ be a multiplicative character of $\F_q$. If some nonzero integral power of $G(\chi)$ is a real number, then we say $G(\chi)$ is pure. 
\end{defn}
We will relate the maximum cliques in a generalized Paley graph with the purity of the corresponding Gauss sums. The following is a simple lemma.

\begin{lem}\label{trivial}
Let $\chi$ be a multiplicative character of  $\F_q$. If $\chi$ is trivial, then $G(\chi)$ is trivial. If $\chi$ is nontrivial and $G(\chi)$ is pure, then $\epsilon(\chi)$ is a root of unity. 
\end{lem}
\begin{proof}
Note that the trivial multiplicative character $\chi_0$ is pure since by Lemma \ref{sumq},
$
G(\chi_0)=\sum_{c \in \F_q} e_p\big( \Tr_{\F_q}(c)\big)=0. 
$
If $\chi$ is nontrivial and $G(\chi)$ is pure, then $|\epsilon(\chi)|=1$ and it must be a root of unity.
\end{proof}

\subsection{Pure Gauss sums and semi-primitive Gauss sums}\label{super}
Chowla \cite{Chowla} proved that if $p$ is an odd prime and $\chi$ a multiplicative character of $\F_p$, then the normalized Gauss sum $\epsilon(\chi)$ is a root of unity if and only if $\chi$ is the quadratic character. A generalization of Chowla's Theorem to any finite fields can be found in \cite[Theorem 4]{Evans}.  It will be helpful if there is a formula to compute Gauss sums so that we can check whether a given Gauss sum is pure. Unfortunately, in general, it is difficult to evaluate Gauss sums explicitly. However, for certain special characters, there are explicit formulas for their associated Gauss sums. 

Stickelberger \cite{Stick} showed that if $-1$ is a power of $p \pmod d$, then all Gauss sums over finite fields of characteristic $p$, corresponding to multiplicative characters with order $d$, are in fact real numbers. These Gauss sums are called {\em semi-primitive}. Semi-primitive Gauss sums are also referred to as uniform cyclotomy or supersingular Gauss sums; see for example \cite{Aoki3, BMW}.
These terminologies are all borrowed from algebraic geometry.

\begin{defn}\label{supers}
Let $q$ be a power of $p$ and let $\chi$ be a multiplicative character of $\F_q$ with order $d$. If $-1$ is a power of $p \pmod d$, then we say $G(\chi)$ is semi-primitive.
\end{defn}

The following theorem is due to Stickelberger \cite{Stick}; see for example \cite[Theorem 11.6.3]{BEW}. The original proof by Stickelberger \cite{Stick} was based on deep algebraic geometry. Simplified proofs were given by Carlitz \cite{LC} and Baumert and McEliece \cite{BM}. 

\begin{thm}[Stickelberger's theorem] \label{forr} 
Let $p$ be an odd prime and let $d>2$ be an integer. Suppose there exists a
positive integer $t$ such that $-1 \equiv p^{t} \pmod d,$ with $t$ chosen minimal. Let $\chi$ be a multiplicative character of $\mathbb{F}_{p^v}$ with order $d$.  Then $v=2ts$ for some positive integer $s$, and
$$
p^{-v/2}G(\chi)=(-1)^{s-1+(p^{t}+1)s/d}.
$$
\end{thm}

One standard way to prove Stickelberger's theorem is to first establish the following special case (see for example \cite[Theorem 5.16]{LN}) , and then using the Hasse-Davenport lifting Theorem (see for example \cite[Theorem 11.5.2]{BEW}) to extend the special case to prove Theorem \ref{forr}.

\begin{thm}[Stickelberger's theorem, special case]\label{stick1}
Let $p$ be an odd prime. Let $q$ be an even power of $p$, and let $\chi$ be a nontrivial multiplicative character of $\F_{q}$ with order $d$ dividing $\sqrt{q}+1$. Then
\begin{equation}\label{stick}
G(\chi)=
\begin{cases}
\sqrt{q} \quad &\text{if $d$ is odd or $\frac{\sqrt{q}+1}{d}$ is even},\\
-\sqrt{q}, \quad & \text{if $d$ is even and $\frac{\sqrt{q}+1}{d}$ is odd}.
\end{cases}    
\end{equation}
\end{thm}

A relatively simple proof of Theorem \ref{stick1} can be found in \cite[Theorem 11.6.3]{BEW}. We will present a new simple proof in Section \ref{main2proof} using Theorem \ref{main2}, which illustrates that results on generalized Paley graphs can be used to establish certain facts on Gauss sums. We will see that the proof of Theorem \ref{main2} only relies on basic tools from Section \ref{Gauss}. 

Note that Theorem \ref{forr} implies that semi-primitive Gauss sums are pure. Evans \cite{Evans2} provided the following necessary and sufficient condition for which a Gauss sum being semi-primitive. 

\begin{thm}[Theorem 1, \cite{Evans2}]\label{power}
Let $d>2$ be an integer such that $d \mid (q-1)$. Let $\chi$ be a multiplicative character of $\F_q$ with order $d$. Then $G(\chi)$ is semi-primitive if and only if $G(\chi^j)$ is pure for all $j \in \N$.
\end{thm}

\section{Gauss sums and maximum cliques in generalized Paley graphs} \label{prep}
The main goal in this section is to present the proof for Theorem \ref{main}. The strategy of the proof goes as follows.
We will first prove Theorem \ref{t1.3}, which recovers the trivial upper bound as we have seen in Section \ref{PGC}. Then we will try to carefully understand the conditions such that we get equality. In particular, we will show a refined version of Theorem \ref{main2} and see that a necessary condition is that the corresponding Gauss sum is pure. We will then use the nice embeddings among generalized Paley graphs to deduce that the corresponding Gauss sum is in fact semi-primitive. This would then allow us to use Stickelberger's Theorem to compute $G(\chi)$ explicitly, which eventually leads to the proof of Theorem \ref{main} in Section \ref{proof}.

For simplicity, we will use $\Tr$ as the shorthand for the absolute trace function $\Tr_{\F_q}$ if no confusion arises.

\subsection{Proof of Theorem \ref{t1.3}}
Before proving Theorem \ref{t1.3}, we need the following lemma.

\begin{lem} \label{sumc}
For any $A \subset \F_q$, we have
\begin{equation} \label{sumceq}
\sum_{c \in \F_q^*} \big|S(q,A;c)\big|^2=
\sum_{c \in \F_q^*} \bigg|\sum_{a\in A} e_p\big(\Tr_{\F_q}(ac)\big)\bigg|^2=q|A|-|A|^2.   
\end{equation}
\end{lem}
\begin{proof}
Expanding the left-hand side of equation \eqref{sumceq}, we get
\begin{align*}
\sum_{c \in \F_q^*} \bigg|\sum_{a\in A} e_p\big(\Tr_{\F_q}(ac)\big)\bigg|^2
&=\sum_{c \in \F_q} \bigg|\sum_{a\in A} e_p\big(\Tr_{\F_q}(ac)\big)\bigg|^2 - |A|^2\\
&=\sum_{c \in \F_q} \sum_{a_1,a_2 \in A} e_p\big(\Tr_{\F_q}(c(a_1-a_2))\big)- |A|^2\\    
&= q|A|- |A|^2+ \sum_{\substack{a_1,a_2 \in A \\a_1 \neq a_2 }}\bigg(\sum_{c \in \F_q} e_p\big(\Tr_{\F_q}(c(a_1-a_2))\big) \bigg)\\
&=q|A|-|A|^2 +\sum_{\substack{a_1,a_2 \in A \\a_1 \neq a_2 }}S(q,\F_q, a_1-a_2).
\end{align*}
By Lemma \ref{sumq}, $S(q,\F_q, a_1-a_2)=0$ for any $a_1,a_2 \in A$ such that $a_1 \neq a_2$. Equation \eqref{sumceq} follows.
\end{proof}

Now we are ready to prove Theorem \ref{t1.3}. 

\begin{proof}[Proof of Theorem \ref{t1.3}] 
By Lemma \ref{qf}, we can write
$$
\sum_{a\in A,\, b\in B}\overline{\chi(a+b)}=\frac{1}{G(\chi)} \sum_{\substack{a\in A,\, b\in B \\ c \in \F_q}} \chi(c) e_p\big(\Tr((a+b)c)\big)
=\frac{1}{G(\chi)} \sum_{c \in \F_q} \chi(c) \sum_{a\in A,\, b\in B} e_p\big(\Tr((a+b)c)\big).
$$
Since $\chi$ is nontrivial, by Lemma \ref{gq}, $|G(\chi)|= \sqrt{q}$. Recall that $|\chi(c)|=1$ for each $c \in \F_q^*$, so by the triangle inequality, we have
\begin{align*}
\bigg|\sum_{a\in A,\, b\in B}\chi(a+b)\bigg| 
&\leq
 \frac{1}{\sqrt{q}}  \sum_{c \in \F_q^*} \bigg|\sum_{a\in A,\, b\in B} e_p\big(\Tr((a+b)c)\big)\bigg|\\
&=\frac{1}{\sqrt{q}}  \sum_{c \in \F_q^*} \bigg|\sum_{a\in A} e_p\big(\Tr(ac)\big)\bigg| \bigg|\sum_{b\in B} e_p\big(\Tr(bc)\big)\bigg|    \\
&\leq \frac{1}{\sqrt{q}}  \bigg(\sum_{c \in \F_q^*} \bigg|\sum_{a\in A} e_p\big(\Tr(ac)\big)\bigg|^2\bigg)^{1/2} \bigg(\sum_{c \in \F_q^*}\bigg|\sum_{b\in B} e_p\big(\Tr(bc)\big)\bigg|^2\bigg)^{1/2}\\
&=\frac{1}{\sqrt{q}} \bigg(\sum_{c \in \F_q^*} \big|S(q,A;c)\big|^2\bigg)^{1/2} \bigg(\sum_{c \in \F_q^*}  \big|S(q,B;c)\big|^2\bigg)^{1/2}.
\end{align*}
where we used the Cauchy-Schwarz inequality in the last inequality. By Lemma \ref{sumc}, 
$$
 \sum_{c \in \F_q^*} \big|S(q,A;c)\big|^2=q|A|-|A|^2, \quad
 \sum_{c \in \F_q^*} \big|S(q,B;c)\big|^2=q|B|-|B|^2.
$$
The required estimate follows.
\end{proof}

\subsection{A refined version of Theorem \ref{t1.3}}

By mimicking the proof of Theorem \ref{t1.3}, we get the following refined version of Theorem \ref{t1.3} in the case that $B=-A$. To be precise, we analyze the situation where Theorem \ref{t1.3} is sharp.

\begin{prop}\label{ga}
Let $p$ be an odd prime and $q$ an even power of $p$. Let $\chi$ be a non-trivial multiplicative character of $\F_q$ and $A$ a subset of $\F_q$ with cardinality $\sqrt{q}$. Then $\chi(a-b)=1$ whenever $a,b \in A$ and $a \neq b$ if and only if for any $c \in \F_q^*$, the complex numbers $\chi(c) |S(q,A;c)|^2$ and $G(\chi)$ share the same argument.
\end{prop}

\begin{proof}
For each $c \in \F_q^*$, let $S(c)=S(q,A;c)$. Similar to the proof of Theorem \ref{t1.3}, we can write
\begin{align*}
\sum_{a,b\in A}\overline{\chi(a-b)}
&=\frac{1}{G(\chi)} \sum_{c \in \F_q^*} \chi(c) \sum_{a,b\in A} e_p\big(\Tr((a-b)c)\big)\\
&=\frac{1}{G(\chi)} \sum_{c \in \F_q^*} \chi(c) \bigg(\sum_{a\in A} e_p\big(\Tr(ac)\big)\bigg) \bigg(\sum_{b\in A} e_p\big(\Tr(-bc)\big)\bigg)    \\
&=\frac{1}{G(\chi)} \sum_{c \in \F_q^*} \chi(c) \bigg(\sum_{a\in A} e_p\big(\Tr(ac)\big)\bigg) \bigg(\sum_{b\in A} e_p\big({-}\Tr(bc)\big)\bigg)    \\
&=\frac{1}{G(\chi)} \sum_{c \in \F_q^*} \chi(c) |S(c)|^2.
\end{align*}
Recall that, by Lemma \ref{sumc}, $\sum_{c \in \F_q^*} |S(c)|^2=q|A|-|A|^2=q\sqrt{q}-q$. 

Suppose $\chi(a-b)=1$ whenever $a,b \in A$ and $a \neq b$, then
\begin{equation} \label{key}
\sum_{c \in \F_q^*} \chi(c) |S(c)|^2=G(\chi)\sum_{a,b\in A}\overline{\chi(a-b)}=G(\chi)(|A|^2-|A|)=G(\chi)(q-\sqrt{q}).   
\end{equation}
Applying the triangle inequality, we obtain
$$
\sum_{c \in \F_q^*} |S(c)|^2=q\sqrt{q}-q=|G(\chi)|(q-\sqrt{q})=\bigg|\sum_{c \in \F_q^*} \chi(c) |S(c)|^2 \bigg| \leq \sum_{c \in \F_q^*} |S(c)|^2.
$$
This is actually equality, so $\chi(c)|S(c)|^2$ must share the same argument (as complex numbers) for all $c \in \F_q^*$, furthermore, by equation \eqref{key}, this argument is also shared by the Gauss sum $G(\chi)$.

Conversely, suppose for any $c \in \F_q^*$, $\chi(c) |S(c)|^2$ and $G(\chi)$ share the same argument. Since $\sum_{c \in \F_q^*} |S(c)|^2>0$, there is $c_0 \in \F_q^*$, such that $S(c_0) \neq 0$. Since $\chi(c_0) |S(c_0)|^2$ and $G(\chi)$ share the same argument, it follows that $\epsilon(\chi)=\chi(c_0)$, and $S(c)=0$ whenever $\chi(c) \neq \chi(c_0)$. Therefore, 
\[
\sum_{c \in \F_q^*} \chi(c) |S(c)|^2
=\chi(c_0)\sum_{\substack{c \in \F_q^*}} |S(c)|^2
=\epsilon(\chi) (q\sqrt{q}-q)
=G(\chi)(q-\sqrt{q}).
\]
Then equation \eqref{key} still holds and $\sum_{a,b\in A}\overline{\chi(a-b)}=|A|^2-|A|$. It follows that $\chi(a-b)=1$ whenever $a,b \in A$ and $a \neq b$.
\end{proof}

\subsection{Proof of Theorem \ref{main2} and its consequences} \label{main2proof}
Similar to the proof of Theorem \ref{trivial_ub}, we can interpret Proposition \ref{ga} in the context of maximum cliques in generalized Paley graphs. We prove the following theorem, which is a refined statement of Theorem \ref{main2}.

\begin{thm} \label{main3}
Let $d$ be a positive integer greater than $1$. Let $q \equiv 1 \pmod {2d}$ be an even power of a prime $p$ and let $\chi$ be a multiplicative character of $\F_q$ with order $d$. Let $A$ be a subset of $\F_q$ with $|A|=\sqrt{q}$. For each $c \in \F_q^*$, let $S(c)=S(q,A;c)$. Then $A$ is a maximum clique in the generalized Paley graph $GP(q,d)$ if and only if for any $c \in \F_q^*$, the complex numbers $\chi(c) |S(c)|^2$ and $G(\chi)$ share the same argument.
In particular, if $\omega\big(GP(q,d)\big)=\sqrt{q}$, then the normalized Gauss sum $\epsilon(\chi)$ is a $d$-th root of unity,  the Gauss sum $G(\chi)=\epsilon(\chi)\sqrt{q}$ is pure, and $S(c)=0$ for any $c \in \F_q^*$ such that $\chi(c)\neq \epsilon(\chi)$.
\end{thm}

\begin{proof}
By the definition of $GP(q,d)$, $A$ forms a (maximum) clique in $GP(q,d)$ if and only if $\chi(a-b)=1$ whenever $a,b \in A$ and $a \neq b$. Therefore, by Proposition \ref{ga}, $A$ is a maximum clique in the generalized Paley graph $GP(q,d)$ if and only if for any $c \in \F_q^*$, the complex numbers $\chi(c) |S(c)|^2$ and $G(\chi)$ share the same argument. 

Next, we assume $\omega\big(GP(q,d)\big)=\sqrt{q}$. By Lemma \ref{sumc}, there is $c_0 \in \F_q^*$ such that $0 \neq \chi(c_0) |S(c_0)|^2$ and $G(\chi)$ share the same argument, i.e. $\chi(c_0)=\epsilon(\chi)$, since $|\epsilon(\chi)|=1$. This implies that $\epsilon(\chi)$ is a $d$-th root of unity. By Lemma \ref{gq}, $G(\chi)=\epsilon(\chi)\sqrt{q}$ is pure. If $c \in \F_q^*$ such that $\chi(c)\neq \epsilon(\chi)$, then $\chi(c)$ and $G(\chi)$ do not share the same argument, and thus $S(c)=0$.
\end{proof}

\begin{rem}\rm
It is standard to show that if $\chi$ has order $d$ and $G(\chi)$ is pure, then $G(\chi^j)$ is pure for any $(j,d)=1$; see for example \cite[Lemma 2.6]{KM}. Thus, if $G(\chi)$ is pure for a character $\chi$ with order $d$, so are other characters  with order $d$. Therefore, in Theorem \ref{main3}, we can take $\chi$ to be any multiplicative character with order $d$.
\end{rem}

Note that Theorem \ref{main3} provides a connection between Gauss sums and the maximum cliques. For the proof of Theorem \ref{main}, we will use facts about Gauss sums. Conversely, it is also interesting to deduce information on Gauss sums using the structure of cliques in the corresponding generalized Paley graphs. In particular, we can produce an elegant and short proof for Stickelberger's Theorem (see Theorem \ref{stick1}).

\begin{proof}[Proof of Theorem \ref{stick1}]
Since $d \mid (\sqrt{q}+1)$ and $\sqrt{q}$ is odd, we have $q \equiv 1 \pmod {2d}$. Thus, the generalized Paley graph $GP(q,d)$ is well-defined. By Theorem \ref{t4}, $\omega\big(GP(q,d)\big)=\sqrt{q}$, and the subfield $A=\F_{\sqrt{q}}$ forms a maximum clique in $GP(q,d)$. 

Let $c \in \F_q^*$ such that $c^{\sqrt{q}-1}=-1$. Then $\Tr_{\F_{q}/A}(c)=c^{\sqrt{q}}+c=0$. By Lemma \ref{sumq}, $S(q,A;c) \neq 0$. Therefore, by Theorem \ref{main3}, $G(\chi)=\epsilon(\chi)\sqrt{q}=\chi(c)\sqrt{q}$. Thus, to compute $G(\chi)$, it suffices to compute $\chi(c)$. Let $g$ be a primitive root of $\F_q$ and write $c=g^k$. Then $-1=c^{\sqrt{q}-1}=g^{k(\sqrt{q}-1)}$ implies that $k(\sqrt{q}-1) \equiv \frac{q-1}{2} \pmod{(q-1)}$, i.e. $k \equiv \frac{\sqrt{q}+1}{2} \pmod{(\sqrt{q}+1)}$. Since $\chi$ has order $d$ and $d \mid (\sqrt{q}+1)$, it follows that 
\[
\chi(c)=\chi(g^k)=\big(\chi(g)\big)^k=\big(\chi(g)\big)^{\frac{\sqrt{q}+1}{2}}=
\begin{cases}
1, \quad &\text{if $\frac{\sqrt{q}+1}{2d} \in \Z$},\\
-1, \quad & \text{otherwise}.
\end{cases}
\]
Note that $\frac{\sqrt{q}+1}{2d} \in \Z$ is equivalent to $d$ is odd or $\frac{\sqrt{q}+1}{d}$ is even, thus equation \eqref{stick} follows immediately.
\end{proof}

\subsection{Exponential-sum-free characterization of maximum cliques} \label{maxclique}
In this section, we illustrate how Theorem \ref{main3} would help give a new proof of Theorem \ref{PS}. For simplicity, we consider the case that $d=2$, i.e., we work on maximum cliques of Paley graphs. We assume that $q=p^{2s}$, and $p^s \equiv 1 \pmod 4$; the case $p^s \equiv 3 \pmod 4$ is similar.

\begin{lem}\label{zerotrace}
If $c$ is a square in $\F_q^*$, then $\Tr_{\F_q/\F_{p^s}}(c) \neq 0$.
\end{lem}
\begin{proof}
Note that $\Tr_{\F_q/\F_{p^s}}(x)=x+x^{p^s}=0$ if and only if $x^{p^s-1}=-1$. Let $g$ be a primitive root of $\F_q$. After writing $x=g^k$, it suffices to show that $k$ is odd. We have $-1=x^{p^s-1}=g^{k(p^s-1)}$, which is equivalent to $k(p^s-1) \equiv \frac{p^{2s}-1}{2} \pmod {p^{2s}-1}$, i.e. $k \equiv \frac{p^s+1}{2} \pmod {p^s+1}$. Since $p^s \equiv 1 \pmod 4$, $\frac{p^s+1}{2}$ is odd and $p^s+1$ is even; thus, $k$ must be odd.
\end{proof}

Using Theorem \ref{main3}, and Lemma \ref{zerotrace}, we obtain the following exponential-sum-free characterization of maximum cliques.

\begin{prop}\label{char}
$A$ is a maximum clique in the Paley graph with order $q=p^{2s}$ if and only if $\{\Tr_{\F_q/\F_{p^s}}(ac): a \in A\}=\F_{p^s}$ for each square $c \in \F_q^*$.
\end{prop}

\begin{proof}
Suppose $\{\Tr_{\F_q/\F_{p^s}}(ac): a \in A\}=\F_{p^s}$ for each square $c \in \F_q^*$. Then for any square $c \in \F_q^*$, by Lemma \ref{trans} and Lemma \ref{sumq}, we have
\[
\sum_{a\in A} e_p\big(\Tr_{\F_q/\F_p}(ac)\big)
=\sum_{a\in A} e_p\big(\Tr_{\F_{p^s}/\F_p}(\Tr_{\F_{q}/\F_{p^s}}(ac))\big)
=\sum_{x \in \F_{p^s}} e_p\big(\Tr_{\F_{p^s}}(x)\big)=0.
\]
Thus Theorem \ref{main3} implies that $A$ is a maximum clique. 

Conversely, assume that $A$ is a maximum clique. Fix a square $c \in \F_q^*$. Then for any $a,b \in A$ with $a \neq b$, $a-b$ is a square in $\F_q^*$, and thus $(a-b)c$ is also a square. By Lemma~\ref{zerotrace}, $\Tr_{\F_q/\F_{p^s}} ((a-b)c) \neq 0$, i.e. $\Tr_{\F_q/\F_{p^s}} (ac) \neq \Tr_{\F_q/\F_{p^s}} (bc)$. Therefore, $\{\Tr_{\F_q/\F_{p^s}}(ac): a \in A\}=\F_{p^s}$.
\end{proof}

We expect that Proposition \ref{char} could be used to produce a purely number-theoretical proof of Theorem \ref{PS}. 

\subsection{Proof of Theorem \ref{main}}\label{proof}

In Theorem \ref{main3}, we showed that the trivial upper bound is tight only if the corresponding Gauss sum is pure. Next, we show a stronger statement.

\begin{prop}\label{stronger}
Let $d>1$ be an integer. Let $q \equiv 1 \pmod {2d}$ be a square and let $\chi$ be a multiplicative character of $\F_q$ with order $d$. If $\omega\big(GP(q,d)\big)=\sqrt{q}$, then $G(\chi^j)$ is pure for any $j \in \N$.
\end{prop}
\begin{proof}
Assume that $\omega\big(GP(q,d)\big)=\sqrt{q}$. By Theorem \ref{main3}. we know that $G(\chi)$ is pure. Fix $j \in \N$. Note that $\chi'=\chi^j$ is a multiplicative character of $\F_q$ with order $d'=\frac{d}{\gcd(d,j)}$. If $d'=1$, then $\chi'$ is the trivial multiplicative character of $\F_q$, and $G(\chi')$ is pure by Lemma \ref{trivial}. Next we assume $d'>1$. Note that $d' \mid d$, so we still have $q \equiv 1 \pmod {2d}$ and $GP(q,d')$ is a well-defined generalized Paley graph. Thus, Lemma \ref{embedding} implies that $\omega\big(GP(q,d')\big)=\sqrt{q}$, and we can use Theorem \ref{main3} to conclude that $G(\chi')$ is also pure.
\end{proof}

Combining Proposition \ref{stronger} and Theorem \ref{power}, we immediately obtain the following corollary.

\begin{cor}\label{semi-primitive}
Let $d>1$ be an integer. Let $q \equiv 1 \pmod {2d}$ be an even power of an odd prime $p$, and let $\chi$ be a multiplicative character of $\F_q$ with order $d$. If the clique number of $GP(q,d)$ is $\sqrt{q}$, then $G(\chi)$ is semi-primitive, equivalently, $-1$ is a power of $p \pmod d$.
\end{cor} 

We have shown that a sufficient condition for $\omega\big(GP(q,d)\big)=\sqrt{q}$ is that the corresponding Gauss sum is semi-primitive. Thus, we can apply Theorem \ref{forr} to get the explicit formula for the associated Gauss sum. Furthermore, in the case that $d$ is odd, we can determine whether $\omega\big(GP(q,d)\big)$ attains the trivial upper bound using the following proposition.

\begin{prop}\label{odd}
Let $p$ be an odd prime and let $d \geq 3$ be an odd integer. Suppose there exists a
positive integer $t$ such that $-1 \equiv p^{t} \pmod d,$ with $t$ chosen minimal. Let $q=p^v \equiv 1 \pmod {2d}$, then $v=2ts$ for some positive integer $s$.  If $s$ is even, then $\omega\big(GP(q,d)\big)<\sqrt{q}$.  If $s$ is odd, then $d \mid (\sqrt{q}+1)$ and $\omega\big(GP(q,d)\big)=\sqrt{q}$.
\end{prop}

\begin{proof}
Let $\chi$ be a multiplicative character of $\mathbb{F}_{p^v}$ with order $d$.  By Theorem \ref{forr}, $v=2ts$ for some positive integer $s$, and
\begin{equation} \label{eq1}
p^{-v/2}G(\chi)=(-1)^{s-1+(p^{t}+1)s/d}.
\end{equation}
Since $d$ is odd and $p^t+1$ is even, equation \eqref{eq1} can be simplified to $\epsilon(\chi)=(-1)^{s-1}$.

If $s$ is even, then $\epsilon(\chi)=-1$ is not a $d$-th root of unity since $d$ is odd. By Theorem \ref{main3}, we have $\omega\big(GP(q,d)\big)<\sqrt{q}$.

If $s$ is odd, then $\sqrt{q}=p^{v/2} \equiv (p^t)^{s} \equiv -1 \pmod d$, i.e. $d \mid (\sqrt{q}+1)$. By Theorem \ref{t4}, we have $\omega\big(GP(q,d)\big)=\sqrt{q}$.
\end{proof}

Now we are ready to prove our main result.

\begin{proof}[Proof of Theorem \ref{main}]
The case that $d$ is even has been proved in Corollary \ref{even}. Next we assume that $d$ is odd. If $d \mid (\sqrt{q}+1)$, then by Theorem \ref{t4}, $\omega\big(GP(q,d)\big)=\sqrt{q}$.

Conversely, suppose $\omega\big(GP(q,d)\big)=\sqrt{q}$. By Corollary \ref{semi-primitive}, $-1$ is a power of $p \pmod d$. By Proposition \ref{odd}, we must have $d \mid (\sqrt{q}+1)$. This completes the proof.
\end{proof}

\section*{Acknowledgement}
The author thanks Shamil Asgarli and Greg Martin for their valuable suggestions, and Lior Silberman,  J\'ozsef Solymosi, and Joshua Zahl for helpful discussions. The research of the author was supported in part by a Four Year Doctoral Fellowship from the University of British Columbia.

\end{document}